\documentclass[leqno, 12pt]{article}
\usepackage{amsmath}
\usepackage{amssymb}
\usepackage{amsthm}

\theoremstyle{plain}
\newtheorem{theorem}{\indent\sc Theorem}[section]\newtheorem{lemma}[theorem]{\indent\sc Lemma}\newtheorem{corollary}[theorem]{\indent\sc Corollary}
\theoremstyle{definition}\newtheorem{definition}[theorem]{\indent\sc Definition}\newtheorem{remark}[theorem]{\indent\sc Remark}

\def\address#1#2{\begingroup\noindent\parbox[t]{7.8cm}{\small{\scshape\ignorespaces#1}\par\vskip1ex\noindent\small{\itshape E-mail}\/: #2\par\vskip4ex}\hfill\endgroup}
\newcommand{\hfl}[2]{\smash{\mathop{\hbox to 9
mm{\rightarrowfill}}\limits^{\scriptstyle #1}_{\scriptstyle #2}}}

\title{\uppercase{On Gras conjecture for imaginary quadratic fields}}
\author{\textsc{Hassan OUKHABA \& St\'ephane VIGUI\'E}}
\begin{document}
\maketitle
\begin{abstract}
In this paper we extend methods of Rubin to prove the Gras conjecture for abelian extensions of a given imaginary quadratic field $k$ and prime numbers $p$ which divide the number of roots of unity in $k$
\end{abstract}
\bibliographystyle{plain}
\section{Introduction}\label{intro}
Let $k\subset\mathbb{C}$ be an imaginary quadratic field, and let $\mathcal{O}_k$ be the ring of integers of $k$. Let $H\subset\mathbb{C}$ be the Hilbert class field of $k$. Let $K\subset\mathbb{C}$ be a finite abelian extension of $k$, and write $\mathrm{G}$ for the Galois group of $K/k$, $\mathrm{G}:=\mathrm{Gal}(K/k)$. Let $\mathcal{O}_K$ and $\mathcal{O}_K^\times$ be the ring of integers of $K$ and the group of units of $\mathcal{O}_K$, respectively. In \cite[Theorem 3.3]{Ru91} Rubin applied the technique of Euler systems to prove the Gras conjecture for $K$, when $H\subset K$ and for all prime number $p$, $p\nmid w_H[K:k]$, where $w_H$ is the number of roots of unity in $H$. Soon after he generalized his result in \cite[Theorem 1]{Ru94} to all $K$ (that is, without the assumption $H\subset K$) and all $p$, $p\nmid w_k[K:k]$, where $w_k$ is the number of roots of unity in $k$. The Gras conjecture is a very subtle information about the ideal class group, used for example in \cite[proof of Theorem 10.3]{Ru91} as the last step in the direction of the main conjecture.\par
In this paper we complete the result of Rubin. Indeed, we prove the Gras conjecture for those prime numbers $p\,\vert w_k$ and $p\nmid[K:k]$. As a first step, we prove a weak form of the Gras conjecture for every prime number $p\nmid[K:k]$. More precisely, Let $\mathcal{E}_K$ be the group of units of $K$ defined in the following section (see the end of section \ref{ES} for a description of the elements of $\mathcal{E}_K$ as elliptic units, and also for a comparison with the group of elliptic units considered by Rubin in \cite{Ru91} and in \cite{Ru94}). Let $\mathrm{Cl}(K)$ be the ideal class group of $K$, and let $\mathrm{g}:=[K:k]$. Then, by applying the elementary approach used in \cite{Vig} (hence without using Euler systems), we prove the following formula for every nontrivial irreducible rational character $\Psi$ of $\mathrm{G}$.
\begin{equation}\label{indice2}
\bigl[e_{\Psi}\bigl(\mathbb{Z}[\mathrm{g}^{-1}]\otimes_{\mathbb{Z}}\mathcal{O}_K^\times\bigr) :e_{\Psi}\bigl(\mathbb{Z}[\mathrm{g}^{-1}]\otimes_{\mathbb{Z}}\mathcal{E}_K\bigr)\bigr]=\#[e_{\Psi}\bigl(\mathbb{Z}[\mathrm{g}^{-1}]\otimes_{\mathbb{Z}}\mathrm{Cl}(K)\bigr)],
\end{equation}  
where $e_\Psi$ is the idempotent of $\mathbb{Z}[\mathrm{g}^{-1}][\mathrm{G}]$ associated to $\Psi$. By $\# X$ we mean the cardinality of the finite set $X$. Of course the formula (\ref{indice2}) is already known if we replace $\mathrm{g}$ by $w_k\mathrm{g}$. This is a consequence of \cite[Theorem 1]{Ru94}. Let us remark that the product of (\ref{indice2}) for all the nontrivial irreducible rational characters of $\mathrm{G}$ yields the equality
\begin{equation}\label{indice2prime}
\bigl[\mathbb{Z}[\mathrm{g}^{-1}]\otimes_{\mathbb{Z}}\mathcal{O}_K^\times :\mathbb{Z}[\mathrm{g}^{-1}]\otimes_{\mathbb{Z}}\mathcal{E}_K\bigr]=\frac{\#\bigl(\mathbb{Z}[\mathrm{g}^{-1}]\otimes_{\mathbb{Z}}\mathrm{Cl}(K)\bigr)}{\#\bigl(\mathbb{Z}[\mathrm{g}^{-1}]\otimes_{\mathbb{Z}}\mathrm{Cl}(k)\bigr)},
\end{equation}
which is known since a long time. Indeed, (\ref{indice2prime}) is a straightforward consequence of \cite[Th\'eor\`eme 5]{Gi79}. It was the ultimate ingredient used by Rubin to prove the Gras conjecture.\par
In the two last sections \ref{ES} and \ref{GC}, we define our Euler systems for $p\,\vert w_k$ and establish all the results needed to apply them to the $p$-part of the ideal class group of $K$. Their application gives Theorem \ref{casfacile}, which, by the help of (\ref{indice2}) (in fact (\ref{indice2prime}) is sufficient), implies  
\begin{theorem}\label{tresgras} Let $p$ be a prime number such that $p\,\vert w_k$ and $p\nmid [K:k]$. Let $\chi$ be a nontrivial irreducible $\mathbb{Z}_p$-character of $\mathrm{G}$. Then
\begin{equation}\label{gras}
\bigl[e_{\chi}\bigl(\mathbb{Z}_p\otimes_{\mathbb{Z}}\mathcal{O}_K^\times\bigr) :e_{\chi}\bigl(\mathbb{Z}_p\otimes_{\mathbb{Z}}\mathcal{E}_K\bigr)\bigr]=\#[e_{\chi}\bigl(\mathbb{Z}_p\otimes_{\mathbb{Z}}\mathrm{Cl}(K)\bigr)],
\end{equation}
where $e_\chi$ is the idempotent of $\mathbb{Z}_p[\mathrm{G}]$ associated to $\chi$.
\end{theorem}
In a forthcoming paper we shall apply Theorem \ref{tresgras} to the main conjecture for prime numbers $p$, $p\,\vert w_k$.
\section{The group $\mathcal{E}_K$}\label{the group}
It is well known that Stark conjectures are satisfied for abelian extensions of imaginary quadratic fields. Moreover, the Stark units are constructed by using appropriate elliptic units. On the other hand, the groups of elliptic units are generated by the norms of these Stark units. We explain this below.\par
For each nonzero ideal $\mathfrak{m}$ of $\mathcal{O}_k$, we denote by $H_{\mathfrak{m}}\subset\mathbb{C}$ the ray class field of $k$ modulo $\mathfrak{m}$. Suppose $\mathfrak{m}\not\in\{(0),\mathcal{O}_k\}$ then Stark proved in \cite{St80} the existence of an element $\varepsilon=\varepsilon_{\mathfrak{m}}\in H_{\mathfrak{m}}$ caracterized, up to a root of unity, by the following three properties
\begin{itemize}
\item[(i)] Let $w_{\mathfrak{m}}$ be the number of roots of unity in $H_{\mathfrak{m}}$.
Then the extension $H_{\mathfrak{m}}(\varepsilon^{1/w_{\mathfrak{m}}})/k$ is abelian.
\item[(ii)] If $\mathfrak{m}$ is divisible by two prime ideals then
$\varepsilon$ is a unit of $\mathcal{O}_{H_{\mathfrak{m}}}$. If $\mathfrak{m}=\mathfrak{q}^e$, where $\mathfrak{q}$ is a prime ideal then 
\begin{equation*}
\varepsilon\mathcal{O}_{H_{\mathfrak{m}}}=(\mathfrak{q})_{\mathfrak{m}}^{\frac{w_{\mathfrak{m}}}{w_k}}\end{equation*}
where $(\mathfrak{q})_{\mathfrak{m}}$ is the product of the prime ideals of $\mathcal{O}_{H_{\mathfrak{m}}}$ which divide $\mathfrak{q}$. 
\item[(iii)] Let $\vert z\vert=z \bar z$ for any complex number $z$, where $\bar z$ is the complex conjugate of $z$. Then
\begin{equation}\label{kronecker}
L'_{\mathfrak{m}}(0,\chi)=
-\frac{1}{w_{\mathfrak{m}}}\sum_{\sigma\in\mathrm{Gal}(H_{\mathfrak{m}}/k)
}\chi(\sigma)\ln\vert\varepsilon^\sigma\vert,
\end{equation}
for all the complex irreducible characters $\chi$ of $\mathrm{Gal}(H_{\mathfrak{m}}/k)$.
\end{itemize}
Here $s\longmapsto L_{\mathfrak{m}}(s,\chi)$ is the
$L$-function associated to $\chi$, defined for the complex numbers
$s$ such that $Re(s)>1$, by the Euler product
\begin{equation*}
L_{\mathfrak{m}}(s,\chi)=
\prod_{\mathfrak{p}\nmid\mathfrak{m}}\bigl(1-\chi(\sigma_{\mathfrak{p}})N(\mathfrak{p})^{-s}\bigr)^{-1},
\end{equation*}
where $\mathfrak{p}$ runs through all prime ideals of $\mathcal{O}_k$ not dividing $\mathfrak{m}$. For such ideal, $\sigma_{\mathfrak{p}}$ and $N(\mathfrak{p})$ are the Frobenius automorphism of $H_{\mathfrak{m}}/k$ and the order of the field $\mathcal{O}_k/\mathfrak{p}$ respectively.\par
For any finite abelian extension $L$ of $k$, we denote by $\mu_L$ the group of roots of unity in $L$, by $w_L$ the order of $\mu_L$ and by $\mathcal{F}_L\subset\mathbb{Z}[\mathrm{Gal}(L/k)]$ the annihilator of $\mu_L$. The description of $\mathcal{F}_L$ given in \cite[page 82, Lemme 1.1]{Tate84} and the property (i) of $\varepsilon_{\mathfrak{m}}$ imply that for any $\eta\in\mathcal{F}_{H_{\mathfrak{m}}}$ there exists $\varepsilon_{\mathfrak{m}}(\eta)\in H_{\mathfrak{m}}$ such that 
\begin{equation*}
\varepsilon_{\mathfrak{m}}(\eta)^{w_{\mathfrak{m}}}=\varepsilon_{\mathfrak{m}}^\eta.
\end{equation*}
\begin{definition}\label{ek} Let $\mathcal{P}_K$ be the subgroup of $K^\times$ generated by $\mu_K$ and by all the norms 
\begin{equation*}
N_{H_{\mathfrak{m}}/H_{\mathfrak{m}}\cap K}(\varepsilon_{\mathfrak{m}}(\eta)),
\end{equation*}
where $\mathfrak{m}$ is any nonzero proper ideal of $\mathcal{O}_k$ and $\eta$ is any element of $\mathcal{F}_{H_{\mathfrak{m}}}$. By definition, 
\begin{equation*}
\mathcal{E}_K:=\mathcal{P}_K\cap\mathcal{O}_K^\times.
\end{equation*}
\end{definition}
We give at the end of section \ref{ES} an other description of $\mathcal{E}_K$ as a group of elliptic units.
\section{The weak Gras conjecture for $\mathcal{E}_K$}\label{WGC}
Let $p\nmid\mathrm{g}$ be a prime number, and let $\mathbb{Z}_{(p)}$ be the localization of $\mathbb{Z}$ at $p$. Let $\mathcal{O}_p$ be the integral closure of $\mathbb{Z}_{(p)}$ in $\mathbb{Q}(\mu_{\mathrm{g}})$. 
Remark that $\mathcal{O}_p$ is a (semi-local) principal ring.
Moreover, if $\zeta\in\mu_{\mathrm{g}}$ is such that $\zeta\neq 1$, then $(1-\zeta)\in\mathcal{O}_p^\times$.
Let us set $v_\infty(z):=-\ln\vert z\vert$ for all $z\in\mathbb{C}^\times$. Let $\ell_K:K^\times\longrightarrow\mathbb{R}[G]$ be the G-equivariant map defined by
\begin{equation*}
\ell_K(x)=\sum_{\sigma\in\mathrm{G}}v_\infty(x^\sigma)\sigma^{-1}.
\end{equation*}
Let $\hat{\mathrm{G}}$ be the group of complex irreducible characters of $\mathrm{G}$.
\begin{lemma}\label{a} Let $\chi\in\hat{\mathrm{G}}$ be such that $\chi\neq1$. Let $\chi_{pr}$ be the character of $\mathrm{Gal}(H_{\mathfrak{f}_\chi}/k)$ defined by $\chi$, where $\mathfrak{f}_\chi$ is the conductor of the fixed field $K_\chi$ of $\ker(\chi)$. Then the following equality holds in $\mathbb{C}[\mathrm{G}]$
\begin{equation}\label{a'}
\mathcal{O}_p\ell_K(\mathcal{E}_K)e_\chi=\mathcal{O}_p\mathcal{F}_K L'_{\mathfrak{f}_\chi}(0,\bar{\chi}_{pr})e_\chi.
\end{equation}
\end{lemma}
\begin{proof} The formula (\ref{a'}) is a direct consequence of (\ref{kronecker}), and can be proved exactly as the formula (3.2) of \cite{Vig}. Indeed, Let $\mathfrak{m}$ be a nonzero proper ideal of $\mathcal{O}_k$ and let $\eta\in\mathcal{F}_{H_{\mathfrak{m}}}$. Let
\begin{equation*}
\varepsilon_{K,\mathfrak{m}}:=N_{H_{\mathfrak{m}}/H_{\mathfrak{m}}\cap K}(\varepsilon_{\mathfrak{m}})\quad\mathrm{and}\quad\varepsilon_{K,\mathfrak{m}}(\eta):=N_{H_{\mathfrak{m}}/H_{\mathfrak{m}}\cap K}(\varepsilon_{\mathfrak{m}}(\eta)). 
\end{equation*}
Then
\begin{equation*}
w_{\mathfrak{m}}\ell_K(\varepsilon_{K,\mathfrak{m}}(\eta))=res^{H_{\mathfrak{m}}}_K(\eta)\ell_K(\varepsilon_{K,\mathfrak{m}}), 
\end{equation*}
where $res^{H_{\mathfrak{m}}}_K:\mathcal{F}_{H_{\mathfrak{m}}}\longrightarrow\mathcal{F}_K$ is the natural restriction map. Since this map is surjective we can proceed now exactly as the proof of the formula (3.2) of \cite{Vig}.
\end{proof}
To prove the formula (\ref{indice2}) we shall use the generalized index of Sinnott, cf. \cite[\S1, page 187]{Si80}. Let $V$ be a $\mathbb{F}$-vector space of finite dimension $d$, where $\mathbb{F}=\mathbb{Q}$ or $\mathbb{F}=\mathbb{R}$. Let $M$ and $N$ be two lattices of $E$, that is two free $\mathbb{Z}$-submodules of $E$, of rank $d$ such that $\mathbb{F}M=\mathbb{F}N=V$. Then we define the index $(M:N)$ by
\begin{equation*}
(M:N)=\vert\det(\gamma)\vert
\end{equation*}
where $\gamma$ is any automorphism of the $\mathbb{F}$-vector space $E$ such that $\gamma(M)=N$. If $N\subset M$ then $(M:N)$ coincides with the usual index $[M:N]$. We also have the following transitivity formula
\begin{equation*}
(M:P)=(M:N)(N:P).
\end{equation*}
This leads to the identity
\begin{equation*}
(M:N)=\frac{[M+N:N]}{[M+N:M]},
\end{equation*}
which may be used as a definition of $(M:N)$. We refer the reader to \cite{Si80} for more details about this generalized index.
\begin{remark} By Dirichlet Theorem we know that $\ell_K(\mathcal{O}_K^\times)$ is a lattice of $\mathbb{R}[\mathrm{G}](1-e_1)$, where $e_1$ is the idempotent associated to the trivial character of $\mathrm{G}$. In particular, for every nontrivial irreducible rational character $\Psi$ of $\mathrm{G}$, the $\mathbb{Z}$-module $\ell_K(\mathcal{O}_K^\times)e_\Psi$ is a lattice of $\mathbb{R}[\mathrm{G}]e_\Psi$. This implies that for every nontrivial $\chi\in\hat{\mathrm{G}}$, the $\mathcal{O}_p$-module $\mathcal{O}_p\ell_K(\mathcal{O}_K^\times)e_\chi$ is free of rank one. Thus, there exists $R_\chi\in\mathbb{C}^\times$ such that
\begin{equation*}
\mathcal{O}_p\ell_K(\mathcal{O}_K^\times)e_\chi= R_\chi\mathcal{O}_pe_\chi.
\end{equation*}
\end{remark} 
\begin{lemma}\label{regulatorun}
Let $\Psi$ be a nontrivial irreducible rational character of $\mathrm{G}$.
Then there exists $u\in\mathcal{O}_p^\times$ such that
\begin{equation}\label{regulateur}
(\mathbb{Z}[\mathrm{G}]e_\Psi : \ell_K(\mathcal{O}_K^\times)e_\Psi)=u \prod_{\chi\vert\Psi} R_\chi.
\end{equation}
\end{lemma}
\begin{proof} Let $\gamma$ be an automorphism of the $\mathbb{R}$-vector space $e_\Psi\mathbb{R}[\mathrm{G}]$ such that $0<\mathrm{det}(\gamma)$ and $\gamma(\mathbb{Z}[\mathrm{G}]e_\Psi) = \ell_K(\mathcal{O}_K^\times)e_\Psi$. Let $(b_\chi)_{\chi\vert\Psi}$ be a $\mathbb{Z}$-basis of $\mathbb{Z}[\mathrm{G}]e_\Psi$. Then $(b_\chi)_{\chi\vert\Psi}$ is an $\mathcal{O}_p$-basis of $\mathcal{O}_p[\mathrm{G}]e_\Psi$. Since $(e_\chi)_{\chi\vert\Psi}$ is also an $\mathcal{O}_p$-basis of $\mathcal{O}_p[\mathrm{G}]e_\Psi$, the automorphism $\beta_1$ of the $\mathbb{C}$-vector space $\mathbb{C}[\mathrm{G}]e_\Psi$, defined by 
\begin{equation*}
\beta_1(b_\chi)=e_\chi,\quad\chi\vert\Psi,
\end{equation*}
is such that $\det(\beta_1)\in\mathcal{O}_p^\times$. In the same manner, since $(\gamma(b_\chi)_{\chi\vert\Psi}$ and $(R_\chi e_\chi)_{\chi\vert\Psi}$ both are $\mathcal{O}_p$-basis of $\mathcal{O}_p\ell_K(\mathcal{O}_K^\times)e_\Psi$, the automorphism $\beta_2$ of $\mathbb{C}[\mathrm{G}]e_\Psi$, defined by 
\begin{equation*}
\beta_2(\gamma(b_\chi))=R_\chi e_\chi,\quad\chi\vert\Psi,
\end{equation*}
is such that $\det(\beta_2)\in\mathcal{O}_p^\times$. Let $\gamma'$ be the automorphism of $\mathbb{C}[\mathrm{G}]e_\Psi$ defined by $\gamma'(e_\chi)=R_\chi e_\chi$, for all $\chi\vert\Psi$. Then $\det(\gamma')=\prod_{\chi\vert\Psi} R_\chi$. Let us extend $u$ to $\mathbb{C}[\mathrm{G}]e_\Psi$ by linearity. Since $(\mathbb{Z}[\mathrm{G}]e_\Psi : \ell_K(\mathcal{O}_K^\times)e_\Psi) = \mathrm{det}(\gamma)$ and  $\gamma\circ\beta_1^{-1}=\beta_2^{-1}\circ\gamma'$ the lemma follows.
\end{proof}
\begin{lemma}
\label{regulatordeux}
Let $F\subseteq K$ be an extension of $k$, and let $R_F$ be the regulator of $F$.
We denote by $\Xi_F$ the set of $\chi\in\hat{\mathrm{G}}$ such that $\chi$ is trivial on $\mathrm{Gal}(K/F)$.
Then, there exists $v\in\mathcal{O}_p^\times$ such that
\begin{equation}\label{decreg}
R_F = v\prod_{\substack{\chi\in\Xi_F\\ \chi\neq1}}R_\chi.
\end{equation}
\end{lemma}
\begin{proof} It is easy to see that $R_F= (I_F : \ell_F(\mathcal{O}_F^\times))$, where $I_F$ is the augmentation ideal of $\mathbb{Z}[\mathrm{Gal}(F/k)]$. Let $D:=\mathrm{Gal}(K/F)$ and $s(D):=\sum\sigma, \sigma\in D$. Let $cor_{K/F}:\mathbb{Z}[\mathrm{Gal}(F/k)]\longrightarrow\mathbb{Z}[\mathrm{G}]$ be the corestriction map. Then, we have
\begin{equation*}
(I_F : \ell_F(\mathcal{O}_F^\times))=(cor_{K/F}(I_F):cor_{K/F}(\ell_F(\mathcal{O}_F^\times)))=
(s(D)I_K : \ell_K(\mathcal{O}_F^\times)).
\end{equation*}
But, the group $\ell_K(\mathcal{O}_F^\times)/s(D)\ell_K(\mathcal{O}_K^\times)$ is finite and annihilated by $\#D$. Thus,
\begin{equation*}
(I_F : \ell_F(\mathcal{O}_F^\times))=w(s(D)I_K : s(D)\ell_K(\mathcal{O}_K^\times)),
\end{equation*}
for some unit $w\in\mathcal{O}_p^\times$. To get the formula (\ref{decreg}) we proceed now exactly as in the proof of Lemma (\ref{regulatorun}).
\end{proof}
\begin{theorem}
Let $\Psi$ be a nontrivial irreducible rational character of $\mathrm{G}$.
Then
\[\left[ e_\Psi \left( \mathbb{Z}\left[\mathrm{g}^{-1}\right] \otimes_{\mathbb{Z}} \mathcal{O}_K^\times \right) : e_\Psi \left( \mathbb{Z}\left[\mathrm{g}^{-1}\right] \otimes_{\mathbb{Z}} \mathcal{E}_K\right) \right] = \#\left( e_\Psi \left( \mathbb{Z}\left[\mathrm{g}^{-1}\right] \otimes_{\mathbb{Z}} \mathrm{Cl}\left(K\right) \right) \right).\]
\end{theorem}
\begin{proof}
Since $K_\chi$ does not depend on the choice of $\chi|\Psi$, let us set $K_\psi:=K_\chi$.
Let $\Xi_\Psi$ be the set of $\chi\in\hat{\mathrm{G}}$ such that $\ker(\chi)$ strictly contains $\mathrm{Gal}(K/K_\Psi)$.
For any $I\subseteq\Xi_\Psi$, we define
\[K_I:=\left\{\begin{array}{ccc}
K_\Psi & \quad\text{if}\quad & I=\varnothing\\
\bigcap_{\chi\in I}F_\chi & \quad\text{if}\quad & I\neq\varnothing.
\end{array}\right.\]
For any $I\subseteq\Xi_\Psi$ let $\zeta_{K_I}$ (resp. $\zeta_k$) be the Dedekind zeta function of $K_I$ (resp. $k$), and let $\tilde{\zeta}_{K_I}(0)$ be the first nonzero coefficient of the Taylor expansion of $\zeta_{K_I}(s)$ at $s=0$.
We also set $\Xi_I:=\Xi_{K_I}$. It is well known that $\zeta_{K_I}$ has a zero of order $[K_I:k]-1$ at $0$, and that for any nontrivial $\chi\in\hat{\mathrm{G}}$, $s\mapsto L_{\mathfrak{f}_\chi}(s,\chi_{pr})$ has a zero of order $1$ at $0$. Then,
\begin{equation}\label{zertaetL}
-\frac{h_IR_I}{w_I}=\tilde{\zeta}_{K_I}(0) = \zeta_k(0)\prod_{\substack{\chi\in\Xi_I\\\chi\neq1}}L'_{\mathfrak{f}_\chi}(0,\chi_{pr}),
\end{equation}
where $h_I:=\#\left(\mathrm{Cl}(K_I)\right)$, $w_I:=\#\left(\mu_{K_I}\right)$, and $R_I$ is the regulator of $K_I$.
For any $\chi\in\hat{\mathrm{G}}$, let $h_\chi\in\mathcal{O}_p$ and $w_\chi\in\mathcal{O}_p$ be such that
\[\mathcal{O}_ph_\chi = Fitt_{\mathcal{O}_p}\left( e_\chi\left(\mathcal{O}_p \otimes_{\mathbb{Z}} \mathrm{Cl}(K) \right)\right) \quad\text{and}\quad \mathcal{O}_pw_\chi = Fitt_{\mathcal{O}_p}\left( e_\chi\left(\mathcal{O}_p \otimes_{\mathbb{Z}} \mu_K \right)\right).
\]
By the inclusion-exclusion principle, as in the proof of \cite[Proposition 3.2]{Vig} we obtain from (\ref{zertaetL}) and Lemma \ref{regulatordeux} the formula
\begin{equation}
\label{decoLprod}
\mathcal{O}_p\prod_{\chi\vert\Psi}L'_{\mathfrak{f}_\chi}(0,\bar{\chi}_{pr}) = \mathcal{O}_p\prod_{\chi\vert\Psi}h_\chi w_\chi^{-1}R_\chi.
\end{equation}
Let us remark that 
\begin{equation}\label{gloougloou}
\mathcal{O}_p\mathcal{F}_Ke_\chi = \mathcal{O}_pw_\chi e_\chi,
\end{equation}
for all $\chi\in\hat{\mathrm{G}}$.
Indeed, since $\mu_K$ is a cyclic $\mathbb{Z}\left[\mathrm{G} \right]$-module, we have $\mathcal{F}_K=Fitt_{\mathbb{Z}\left[\mathrm{G} \right]}\left(\mu_K \right)$, and then
\[\mathcal{O}_p\mathcal{F}_K = Fitt_{\mathcal{O}_p\left[\mathrm{G} \right] } \left(\mathcal{O}_p\otimes_{\mathbb{Z}}\mu_K\right) = \mathop{\oplus}_{\chi\in\hat{\mathrm{G}}} Fitt_{\mathcal{O}_p} \left(e_\chi \left(\mathcal{O}_p\otimes_{\mathbb{Z}}\mu_K\right)\right)e_\chi, \]
which implies (\ref{gloougloou}).
Therefore, proceeding as in the proof of Lemma \ref{regulatorun}, one can show from Lemma \ref{a} and (\ref{gloougloou}) that
\begin{equation}\label{ZGlE}
\mathcal{O}_p \left( \mathbb{Z}\left[\mathrm{G}\right]e_\Psi : \ell_K(\mathcal{E}_K)e_\Psi\right) = \mathcal{O}_p\prod_{\chi\vert\Psi} w_\chi L'_{\mathfrak{f}_\chi}(0,\bar{\chi}_{pr}).
\end{equation}
From (\ref{ZGlE}) and (\ref{decoLprod}), we obtain
\begin{equation}\label{ZGlEdeux}
\mathcal{O}_p \left( \mathbb{Z}\left[\mathrm{G}\right]e_\Psi : \ell_K(\mathcal{E}_K)e_\Psi\right) = \mathcal{O}_p\prod_{\chi\vert\Psi}h_\chi R_\chi = \mathcal{O}_p \#\left( e_\Psi \left( \mathbb{Z}_{(p)} \otimes_{\mathbb{Z}} \mathrm{Cl}(K) \right) \right)\prod_{\chi\vert\Psi}R_\chi.
\end{equation}
From (\ref{ZGlEdeux}) and Lemma (\ref{regulatorun}), we have
\begin{equation}
\label{idx}
\mathcal{O}_p\left[ \ell_K(\mathcal{O}_K^\times)e_\Psi : \ell_K(\mathcal{E}_K) e_\Psi\right] = \mathcal{O}_p \#\left( e_\Psi \left(\mathbb{Z}_{(p)} \otimes_{\mathbb{Z}} \mathrm{Cl}(K) \right) \right).
\end{equation}
Since $p$ is prime to $\mathrm{g}$, (\ref{idx}) gives
\begin{equation}
\label{idxdeux}
\mathcal{O}_p\left[ \mathbb{Z}\left[\mathrm{g}^{-1}\right] \ell_K(\mathcal{O}_K^\times)e_\Psi : \mathbb{Z}\left[\mathrm{g}^{-1}\right] \ell_K(\mathcal{E}_K) e_\Psi\right] = \mathcal{O}_p \#\left( e_\Psi \left( \mathbb{Z}\left[\mathrm{g}^{-1}\right] \otimes_{\mathbb{Z}} \mathrm{Cl}(K) \right) \right).
\end{equation}
This being true for every prime $p\nmid\mathrm{g}$, and since the integers we are comparing are prime to $\mathrm{g}$, we have
\begin{equation}
\label{idxtrois}
\left[ \mathbb{Z}\left[\mathrm{g}^{-1}\right] \ell_K(\mathcal{O}_K^\times)e_\Psi : \mathbb{Z}\left[\mathrm{g}^{-1}\right] \ell_K(\mathcal{E}_K) e_\Psi\right] = \#\left( e_\Psi \left( \mathbb{Z}\left[\mathrm{g}^{-1}\right] \otimes_{\mathbb{Z}} \mathrm{Cl}(K) \right) \right).
\end{equation}
But $\mathcal{O}_K^\times/\mathcal{E}_K\simeq\ell_K(\mathcal{O}_K^\times)/\ell_K(\mathcal{E}_K)$, and the theorem follows.
\end{proof}
\section{The Euler system}\label{ES}
For any finite abelian extension $F$ of $k$, and any fractional ideal $\mathfrak{a}$ of $k$ prime to the conductor of $F/k$, we denote by $(\mathfrak{a}, F/k)$ the automorphism of $F/k$ associated to $\mathfrak{a}$ by the Artin map. If $\mathfrak{a}\subset\mathcal{O}_k$ then we denote by $N(\mathfrak{a})$ the cardinality of $\mathcal{O}_k/\mathfrak{a}$. Let $\mathcal{I}$ be the group of fractional ideals of $k$ and let us consider its subgroup
$\mathcal{P}:=\{x\mathcal{O}_k,\ x\in k^\times\}$. Let $H:=H_{(1)}\subset\mathbb{C}$ be the Hilbert class field of $k$. Then, the Artin map gives an isomorphism from $\mathrm{Cl}(k):=\mathcal{I}/\mathcal{P}$ into Gal$(H/k)$. Let $p$ be a prime number such that $p\vert w_k$, and let $\mathrm{Cl}_p(k)$ be the $p$-part of $\mathrm{Cl}(k)$. Then, fix $\mathfrak{a}_1,\ldots,\mathfrak{a}_s$, a finite set of ideals of $\mathcal{O}_k$ such that
\begin{equation}\label{pic}
\mathrm{Cl}_p(k)=<\bar{\mathfrak{a}}_1>\times\cdots\times<\bar{\mathfrak{a}}_s>,
\end{equation}
where $<\bar{\mathfrak{a}}_i>\neq1$ is the group generated by the class $\bar{\mathfrak{a}}_i$ of $\mathfrak{a}_i$ in $\mathrm{Cl}(k)$. If $n_i$ is the order of $<\bar{\mathfrak{a}}_i>$, then $(\mathfrak{a}_i)^{n_i}=a_i\mathcal{O}_k$, with $a_i\in\mathcal{O}_k$. If $\mathrm{Cl}_p(k)=1$ then we set $s=1$, $\mathfrak{a}_1=\mathcal{O}_k$ and $a_1=1$.\par  
Let $p\vert w_k$ be a prime number as above, and let $M$ be a power of $p$. Let $\mu_M$ be the group of $M$-th roots of unity in $\mathbb{C}$. Then we define
\begin{equation}\label{kaem}
K_M:=K(\mu_M,(\mathcal{O}_k^\times)^{1/M}).
\end{equation}
Moreover, we denote by $\mathcal{L}$ the set of prime ideals $\ell$ of $\mathcal{O}_k$ such that $\ell$ splits completely in the Galois extension $K_M\bigl(a_1^{1/M},\ldots,a_s^{1/M}\bigr)/k$.
Exactly as in \cite[Lemma 3]{Ru94} or in \cite[Lemma 3.1]{ou13vig2} we have
\begin{lemma}\label{extension} For each prime $\ell\in\mathcal{L}$ there exists a cyclic extension $K(\ell)$ of $K$ of degree $M$, contained in the compositum $K.H_{\ell}$, unramified outside $\ell$, and such that $K(\ell)/K$ is totally ramified at all primes above $\ell$.
\end{lemma}
\begin{proof} See for instance the proof of Lemma 3.1 of \cite{ou13vig2}
\end{proof}
Let $\mathcal{S}$ be the set of squarefree ideals of $\mathcal{O}_k$ divisible only by primes $\ell\in\mathcal{L}$. If $\mathfrak{a}=\ell_1\cdots\ell_n\in\mathcal{S}$ then we set $K(\mathfrak{a}):=K(\ell_1)\cdots K(\ell_n)$ and $K(\mathcal{O}_k):=K$. If $\mathfrak{g}$ is an ideal of $\mathcal{O}_k$ then we denote by $\mathcal{S}(\mathfrak{g})$ the set of ideals $\mathfrak{a}\in\mathcal{S}$ that are prime to $\mathfrak{g}$. Following Rubin we define an Euler system to be a function 
\begin{equation*}
\alpha:\mathcal{S}(\mathfrak{g})\longrightarrow \mathbb{C}^\times,
\end{equation*}
such that
\begin{description}
\item{E1.}\ $\alpha(\mathfrak{a})\in K(\mathfrak{a})^\times$.
\item{E2.}\ $\alpha(\mathfrak{a})\in\mathcal{O}_{K(\mathfrak{a})}^\times$, if $\mathfrak{a}\neq\mathcal{O}_k$.
\item{E3.}\  $N_{K(\mathfrak{a}\ell)/K(\mathfrak{a})}\bigl(\alpha(\mathfrak{a}\ell)\bigr)=\alpha(\mathfrak{a})^{1-\mathrm{Fr}(\ell)^{-1}}$, where Fr$(\ell)$ is the Frobenius of $\ell$ in Gal$(K(\mathfrak{a})/k)$.
\item{E4.}\ $\alpha(\mathfrak{a}\ell)\equiv\alpha(\mathfrak{a})^{\mathrm{Frob}(\ell)^{-1}(N(\ell)-1)/M}$ modulo all primes above $\ell$.
\end{description}\par
For the convenience of the reader we recall now the construction of Euler systems by using elliptic units. To this end we use the elliptic functions $\Psi(.;L,L'):z\longmapsto\Psi(z;L,L')$ introduced by G.~Robert in \cite{Rob90} and \cite{Rob92}, where $L\subset L'$ are lattices of $\mathbb{C}$ such that the index $[L':L]$ is prime to $6$.
As proved by Robert, for instance in \cite{Rob89} and \cite{Rob91}, if $\mathfrak{m}$ is a nonzero proper ideal of $\mathcal{O}_k$ and $\mathfrak{g}$ is an ideal of $\mathcal{O}_k$ prime to $6\mathfrak{m}$ then $\Psi(1;\mathfrak{m},\mathfrak{g}^{-1}\mathfrak{m})\in H_{\mathfrak{m}}$. Let us denote by $r_{\mathfrak{m}}$ the order of the kernel of the natural map $\mu_k\longrightarrow(\mathcal{O}_k/\mathfrak{m})^\times$ and let $\mathfrak{m}'$ be a nonzero proper ideal of $\mathcal{O}_k$ such that $\mathfrak{m}\vert\mathfrak{m}'$, $\mathfrak{m}'$ is divisible by the same prime ideals that divide $\mathfrak{m}$ and $r_{\mathfrak{m}'}=1$, then
\begin{equation*}
N_{H_{\mathfrak{m}'}/H_{\mathfrak{m}}}(\Psi(1;\mathfrak{m}',\mathfrak{g}^{-1}\mathfrak{m}'))^{w_{\mathfrak{m}}}=\varepsilon_{\mathfrak{m}}^{N(\mathfrak{g})-(\mathfrak{g},\, H_{\mathfrak{m}}/k)}.
\end{equation*}
In particular $\mathcal{P}_K$ is generated as an abelian group by $\mu_K$ and by all the norms
\begin{equation*}
\Psi_{\mathfrak{m}}(\mathfrak{g}):=N_{H_{\mathfrak{m}}/H_{\mathfrak{m}}\cap K}(\Psi(1;\mathfrak{m},\mathfrak{g}^{-1}\mathfrak{m})),
\end{equation*}
where $\mathfrak{m}$ and $\mathfrak{g}$ are any nonzero ideals of $\mathcal{O}_k$ such that $\mathfrak{m}\neq\mathcal{O}_k$ and $\mathfrak{g}$ is prime to $6\mathfrak{m}$. If $\mathfrak{m}=\mathfrak{nq}$, where $\mathfrak{n}\neq\mathcal{O}_k$ and $\mathfrak{q}$ is a prime ideal of $\mathcal{O}_k$, then
\begin{equation*}
N_{H_{\mathfrak{m}}/H_{\mathfrak{n}}}(\Psi(1;\mathfrak{m},\mathfrak{g}^{-1}
\mathfrak{m}))^{\frac{r_{\mathfrak{n}}}{r_{\mathfrak{m}}}}=\left\lbrace
\begin{array}{cc}
\Psi(1;\mathfrak{n},\mathfrak{g}^{-1}
\mathfrak{n})&\mathrm{if}\ \mathfrak{q}\vert\mathfrak{n}\\
 & \\
\Psi(1;\mathfrak{n},\mathfrak{g}^{-1}
\mathfrak{n})^{1-(\mathfrak{q}, H_{\mathfrak{n}}/k)^{-1}}&\mathrm{si}\ \mathfrak{q}\nmid\mathfrak{n}.
\end{array}\right.
\end{equation*}
Moreover, if $\mathfrak{q}\nmid\mathfrak{n}$ then $\Psi(1;\mathfrak{m},\mathfrak{g}^{-1}
\mathfrak{m})^{N(\mathfrak{q})}\equiv\Psi(1;\mathfrak{n},\mathfrak{g}^{-1}\mathfrak{n})$ modulo all primes above $\mathfrak{q}$. Therefore, the map $\alpha:\mathcal{S}(\mathfrak{mg})\longrightarrow\mathbb{C}^\times$, defined by
\begin{equation*}
\alpha(\mathfrak{a}):=N_{KH_{\mathfrak{ma}}/K(\mathfrak{a})}(\Psi(1;\mathfrak{ma},\mathfrak{g}^{-1}\mathfrak{ma})),
\end{equation*}
is an Euler system satisfying $\alpha(1)=\Psi_{\mathfrak{m}}(\mathfrak{g})$. In particular we have
\begin{corollary}\label{sardine}If $u\in\mathcal{E}_K$ then there exists an ideal $\mathfrak{f}$ of $\mathcal{O}_k$ and an Euler system $\alpha:\mathcal{S}(\mathfrak{f})\longrightarrow \mathbb{C}^\times$, such that $\alpha(1)=u$
\end{corollary}
\begin{proof} In view of the discussion above we only have to check the corollary for the roots of unity in $K$. We leave this as an exercise or see \cite[proof of Proposition 1.2]{Ru91}.
\end{proof}
\begin{remark} We recall that the group of elliptic units considered by Rubin in \cite{Ru91} and \cite{Ru94} is the subgroup of $\mathcal{O}_K^\times$ generated by $\mu_K$ and by all $\Psi_{\mathfrak{m}}(\mathfrak{g})^{\sigma-1}$, where $\sigma\in\mathrm{G}$ and $\mathfrak{m}$ and $\mathfrak{g}$ are as above. Let us denote this group by $C_K$. It is clear that $C_K\subset\mathcal{E}_K$ and $(\mathcal{E}_K)^{\mathrm{g}}\subset C_K$, and then $$\mathbb{Z}[\mathrm{g}^{-1}]\otimes_{\mathbb{Z}}\mathcal{E}_K=\mathbb{Z}[\mathrm{g}^{-1}]\otimes_{\mathbb{Z}}C_K.$$ 
\end{remark}

\section{The Gras conjecture}\label{GC}
Exactly as in \cite[Proposition 2.2]{Ru91}, one can prove that for any Euler system $\alpha:\mathcal{S}(\mathfrak{g})\longrightarrow \mathbb{C}^\times$ there is a natural map
\begin{equation}\label{kappa}
\kappa_\alpha:\mathcal{S}(\mathfrak{g})\longrightarrow K^\times/(K^\times)^M,\quad \kappa_\alpha(\mathfrak{a})\equiv\alpha(\mathfrak{a})^{D_{\mathfrak{a}}}\ \mathrm{modulo}\ (K^\times)^M.   
\end{equation}
Let $\mathcal{I}=\oplus_\lambda\mathbb{Z}\lambda$ be the group of fractional ideals of $K$ written additively. If $\ell$ is a prime ideal of $\mathcal{O}_k$ then we define $\mathcal{I}_\ell:=\oplus_{\lambda\vert\ell}\mathbb{Z}\lambda$. If $y\in K^\times$ then we denote by $(y)_\ell\in\mathcal{I}_\ell$, $[y]\in\mathcal{I}/M\mathcal{I}$ and $[y]_\ell\in\mathcal{I}_\ell/M\mathcal{I}_\ell$ the projections of the fractional ideal $(y):=y\mathcal{O}_K$. Let us suppose that $\ell\in\mathcal{L}$. Let $\lambda'$ be a prime ideal of $\mathcal{O}_{K(\ell)}$ above $\ell$, and let $\pi\in\lambda'-(\lambda')^2$. Then $\pi^{1-\sigma_\ell}$ has exact order $M$ in the cyclic group $(\mathcal{O}_{K(\ell)}/\lambda')^\times$, because $K(\ell)/K$ is cyclic, totally ramified at $\lambda:=\lambda'\cap\mathcal{O}_K$. In particular, using the isomorphism $\mathcal{O}_{K(\ell)}/\lambda'\simeq\mathcal{O}_K/\lambda$, there exists $x_\lambda\in(\mathcal{O}_K/\lambda)^\times$ such that the image of $\pi^{1-\sigma_\ell}$ in $(\mathcal{O}_K/\lambda)^\times$ is equal to $(x_\lambda)^d$, where $d:=(N(\ell)-1)/M$. Let us remark that the projection of $x_\lambda$ in $(\mathcal{O}_K/\lambda)^\times/((\mathcal{O}_K/\lambda)^\times)^M$ is well defined, does not depend on $\pi$ and, in fact, has exact order $M$. Thus, the isomorphism $\mathcal{O}_K/\ell\mathcal{O}_K\simeq\oplus_{\lambda\vert\ell}\mathcal{O}_K/\lambda$ allows us to define a $\mathrm{G}$-equivariant isomorphism
\begin{equation}\label{fielle}
\hat{\varphi}_\ell:(\mathcal{O}_K/\ell\mathcal{O}_K)^\times/((\mathcal{O}_K/\ell\mathcal{O}_K)^\times)^M\longrightarrow\mathcal{I}_\ell/M\mathcal{I}_\ell,   
\end{equation}
such that the image of an element $x:=\oplus_{\lambda\vert\ell}(x_\lambda)^{e_\lambda}$ is $\hat{\varphi}_\ell(x):=\oplus_{\lambda\vert\ell}e_\lambda\lambda$. Let us consider the map 
\begin{equation}\label{requin}
\psi_\ell:K(\ell)^\times\longrightarrow(\mathcal{O}_K/\ell\mathcal{O}_K)^\times/((\mathcal{O}_K/\ell\mathcal{O}_K)^\times)^M,   
\end{equation}
which associates to $z$ the sum $\oplus_{\lambda\vert\ell}z_\lambda$ such that the image of $z^{1-\sigma_\ell}$ in $(\mathcal{O}_K/\lambda)^\times$ is equal to $(z_\lambda)^d$. Let $\varphi_\ell:=-\hat{\varphi}_\ell$, then
\begin{equation}\label{baleine}
(\varphi_\ell\circ\psi_\ell)(x)=[N_{K(\ell)/K}(x)]_\ell.   
\end{equation}
The map $\varphi_\ell$ induces a homomorphism $\{y\in K^\times/(K^\times)^M,\ [y]_\ell=0\}\longrightarrow\mathcal{I}_\ell/M\mathcal{I}_\ell$ which we also denote by $\varphi_\ell$. Then, as in \cite[Proposition 2.4]{Ru91}, one can prove that for any Euler system $\alpha:\mathcal{S}(\mathfrak{g})\longrightarrow\mathbb{C}^\times$, and any $\mathfrak{a}\in\mathcal{S}(\mathfrak{g})$, such that $\mathfrak{a}\neq1$
\begin{equation}\label{marrakech}
[\kappa_\alpha(\mathfrak{a})]_\ell=
\begin{cases}
0&\mathrm{if}\ \ell\nmid\mathfrak{a}\\
\varphi_\ell(\kappa_\alpha(\mathfrak{a}/\ell))&\mathrm{if}\ \ell\vert\mathfrak{a}.
\end{cases}
\end{equation}
In the sequel, if $p$ is a prime number such that $p\nmid[K:k]$, $\chi$ a nontrivial irreducible $\mathbb{Z}_p$-character of $\mathrm{G}$, and $\Pi$ is a $\mathbb{Z}_p[\mathrm{G}]$-module then we define $\Pi_\chi:=e_\chi\Pi$. If $\Pi$ is a $\mathbb{Z}[\mathrm{G}]$-module then we define $\Pi_\chi:=e_\chi(\mathbb{Z}_p\otimes_{\mathbb{Z}}\Pi)$. Before proving Theorem \ref{tresgras} we need first prove the analoguous of \cite[Theorem 4]{Ru94} and \cite[Theorem 3.1]{Ru91}. For this, if $p\,\vert w_k$ is a prime number and $M$ is a $p$-power, then we set
\begin{equation*}
K':=K_M(a_1^{1/M},\ldots,a_s^{1/M}).
\end{equation*}
\begin{lemma}\label{fund}
Let $p$ be a prime number such that $p\,\vert w_k$, and let $M\geq p$ be a power of $p$. Let us consider the natural map
\begin{equation*}
\Theta:K^\times/(K^\times)^M\longrightarrow K_M^\times/(K_M^\times)^M.
\end{equation*}
\begin{description}
\item(i) If $p=3$ or ($p=2$ and $\mu_4\subset K$) or $M=2$ then $\ker(\Theta)=\mathcal{O}_{k}^{\times}/(\mathcal{O}_{k}^{\times})^M$.
\item(ii) If $4\,\vert M$ and $w_k\in\{2,6\}$ then $\ker(\Theta)$ is generated by the projections in  $K^\times/(K^\times)^M$ of $\mathcal{O}_{k}^{\times}$ and $2^{M/2}$
\end{description}
In particular, $\ker(\Theta)$ is annihilated by $[K:k]-s(\mathrm{G})$, where $s(\mathrm{G}):=\sum\sigma, \sigma\in\mathrm{G}$.
Furthermore, since $K'/K_M$ is a Kummer extension and $a_1,\ldots,a_s$ are elements of $k$, the kernel of the natural map
\begin{equation*}
K^\times/(K^\times)^M\longrightarrow K'^\times/(K'^\times)^M
\end{equation*}
is also annihilated by $[K:k]-s(\mathrm{G})$.
\end{lemma}
\begin{proof} Let $x\in K^\times\cap(K_M^\times)^M$. If $p=3$ then $K_M=K(\mu_{3M})$. By \cite[Lemma 5.7 (i)]{Ru87} we have $x^3\in(K^\times)^{3M}$, that is to say, $x\in\mu_3(K^\times)^{M}\subset\mu_k(K^\times)^{M}$. If $p=2$ and $w_k=4$ then $K_M=K(\mu_{4M})$. Again by \cite[Lemma 5.7 (i)]{Ru87} we deduce that $x\in\mu_k(K^\times)^{M}$. Suppose now that $p=2$ and $w_k\in\{2,6\}$. Then $K_M=K(\mu_{2M})$. By using the same arguments as before, we see that $x\in\mu_k(K(\mu_4)^\times)^{M}$. If $\mu_4\subset K$ or $M=2$ we are done. Let us assume that $\mu_4\not\subset K$ and $4\vert M$; and write $x=z^M\zeta$, for some $z\in K(\mu_4)$ and $\zeta\in\mu_k$. Let $\sigma$ be the unique nontrivial automorphism of $K(\mu_4)/K$. If $z^{\sigma-1}\in\mu_2$ then it is easy to check that $x\in\mu_k(K^\times)^{M}$. Suppose we are in the case $z^{\sigma-1}=i$, where $i^2=-1$. Since $\sigma(i)=-i$ and $z=a+ib$, where $a,b\in K$, the equation $a-ib=\sigma(z)=i(a+ib)$ implies that $b=-a$, $z=a(1-i)$ and $x=a^M2^{M/2}(-1)^{M/4}\zeta$. The complex number $\zeta_8:=(1+i)\sqrt{2}/2$ is a root of unity of order $8$. An easy computaion shows that we can not have $z^{\sigma-1}=\zeta_8$. This proves the assertions $(i)$ and $(ii)$. The reste of the lemma is straightforward.
\end{proof}
\begin{lemma}\label{capital} Suppose $p$ is a prime number such that $p\nmid[K:k]$ and $p\vert w_k$. Let $M$ be a power of $p$. Let $\chi$ be a nontrivial irreducible $\mathbb{Z}_p$-character of $\mathrm{G}$. Let $\mathsf{H}^\chi$ be the abelian extension of $K$ corresponding to the $\chi$-part $\mathrm{Cl}(K)_\chi$. Then $\mathsf{H}^\chi\cap K'=K$.
\end{lemma}
\begin{proof} The group $\mathrm{G}$ acts trivially on $\mathrm{Gal}(\mathsf{H}^\chi\cap K_M/K)$ because $K_M$ is abelian over $k$. On the other hand, $\mathrm{Gal}(\mathsf{H}^\chi\cap K_M/K)$ is a $\mathrm{G}$-quotient of $\mathrm{Gal}(\mathsf{H}^\chi/K)\simeq\mathrm{Cl}(K)_\chi$. This implies that $\mathsf{H}^\chi\cap K_M=K$ since $\chi\neq1$. In addition, if $p\nmid[H:k]$ then $K'=K_M$. In particular we have proved that $\mathsf{H}^\chi\cap K'=K$ in case $p\nmid[H:k]$. Let $E:=K_M(\mathsf{H}^\chi\cap K')$. By Kummer theory we deduce from the inclusion $E\subset K'$ that $E=K_M(V^{1/M})$, where $V$ is a subgroup of the multiplicative group $<a_1,\ldots,a_s>\subset k^\times$. If $p\,\vert[H:k]$ then $\mathrm{G}$ acts on $\mathrm{Gal}(E/K_M)$ via the trivial character. This implies that $\mathrm{Gal}(E/K_M)=1$ because this group is isomorphic to $\mathrm{Gal}(\mathsf{H}^\chi\cap K'/K)$ on which $\mathrm{G}$ acts via $\chi\neq1$. The proof of the lemma is now complete.
\end{proof}
\begin{theorem}\label{Chebotarev} Suppose $p$ is a prime number such that $p\nmid[K:k]$ and $p\,\vert w_k$. Let $M$ be a power of $p$. Let $\chi$ be a nontrivial irreducible $\mathbb{Z}_p$-character of $\mathrm{G}$. Let $\beta\in (K^\times/(K^\times)^M)_\chi$ and $A$ be a $\mathbb{Z}_p[\mathrm{G}]$-quotient of $\mathrm{Cl}(K)_\chi$. Let $m$ be the order of $\beta$ in $K^\times/(K^\times)^M$, $W$ the $\mathrm{G}$-submodule of $K^\times/(K^\times)^M$ generated by $\beta$, $\mathsf{H}$ the abelian extension of $K$ corresponding to $A$, and $L:=\mathsf{H}\cap K'(W^{1/M})$. Then, there is a $\mathbb{Z}[\mathrm{G}]$ generator $\mathfrak{c}'$ of $\mathrm{Gal}(L/K)$ such that for any $\mathfrak{c}\in A$ whose restriction to $L$ is $\mathfrak{c}'$, there are infinitely many prime ideals $\lambda$ of $\mathcal{O}_K$ of degree one such that
\begin{description}
\item(i) the projection of the class of $\lambda$ in $A$ is $\mathfrak{c}$,
\item(ii) if $\ell:=\lambda\cap\mathcal{O}_k$ then $\ell\in\mathcal{L}$,
\item(iii) $[\beta]_\ell=0$ and there is $u\in(\mathbb{Z}/M\mathbb{Z}[\mathrm{G}])_\chi^\times$ such that $\varphi_\ell(\beta)=(M/m)u\lambda$.
\end{description} 
\end{theorem}
\begin{proof} We follow \cite[Theorem 3.1]{Ru91}. Since $W\subset (K^\times/(K^\times)^M)_\chi$ and $\chi\neq1$, we deduce from Lemma \ref{fund} that the Galois group of the Kummer extension $K'(W^{1/M})/K'$ is isomorphic as a $\mathbb{Z}[\mathrm{Gal}(K_M/k)]$-module to $\mathrm{Hom}(W,\mu_M)$. But $W\simeq(\mathbb{Z}/m\mathbb{Z}[\mathrm{G}])_\chi$, which is a direct factor of $(\mathbb{Z}/m\mathbb{Z})[\mathrm{G}]$.
On the other hand, $\mathrm{Hom}((\mathbb{Z}/m\mathbb{Z})[\mathrm{G}],\mu_M)$ is $\mathbb{Z}[\mathrm{Gal}(K_M/k)]$-cyclic, generated for instance by the group homomorphism $\Psi:(\mathbb{Z}/m\mathbb{Z})[\mathrm{G}]\longrightarrow\mu_M$ defined by $\Psi(1)=\zeta$ and $\Psi(g)=1$, for $g\neq1$, where $\zeta\in\mu_M$ is a primitive $m$-th root of unity. Therefore, we can find $\tau\in\mathrm{Gal}(K'(W^{1/M})/K')$ which generates $\mathrm{Gal}(K'(W^{1/M})/K')$ over
$\mathbb{Z}[\mathrm{Gal}(K_M/k)]$. The restriction $\mathfrak{c}'$ of $\tau$ to $L$ is a  $\mathbb{Z}[\mathrm{G}]$ generator of $\mathrm{Gal}(L/K)\simeq\mathrm{Gal}(LK'/K')$ by Lemma \ref{capital}. Let $\mathfrak{c}\in\mathrm{Gal}(\mathsf{H}/K)=A$ be any extension of $\mathfrak{c}'$ to $\mathsf{H}$. Then one can find $\sigma\in\mathrm{Gal}(\mathsf{H}K'(W^{1/M})/K)$ such that
\begin{equation*}
\sigma_{\vert\mathsf{H}}=\mathfrak{c}\quad\mathrm{and}\quad\sigma_{\vert K'(W^{1/M})}=\tau.
\end{equation*}
By Chebotarev density theorem there exist infinitely many primes $\lambda$ of $\mathcal{O}_K$ whose Frobenius in $\mathrm{Gal}(\mathsf{H}K'(W^{1/M})/K)$ is the congugacy class of $\sigma$, and such that $\ell:=\lambda\cap\mathcal{O}_k$ is unramified in $K'(W^{1/M})/k$. Now it is immediate that $(i)$ and $(ii)$ are satisfied. The rest of the proof is exactly the same as the proof of \cite[Theorem 3.1]{Ru91}.
\end{proof}
\begin{theorem}\label{casfacile} Suppose $p$ is a prime number such that $p\nmid[K:k]$ and $p\,\vert w_k$. Let $\chi$ be a nontrivial irreducible $\mathbb{Z}_p$-character of $\mathrm{G}$. Then we have
\begin{equation}\label{division}
\#\mathrm{Cl}(K)_\chi\,\vert\,\#(\mathcal{O}_K^\times/\mathcal{E}_K)_\chi.   
\end{equation}
\end{theorem}
\begin{proof} We proceed exactly as \cite[proof of Theorem 3.2]{Ru91} or \cite[proof of Theorem 4.4]{ou13vig2}. Let $\hat{\chi}$ be a $\overline{\mathbb{Q}}_p$-irreducible character of $\mathrm{G}$ such that $\hat{\chi}\vert\chi$, and let $\hat{\chi}(\mathrm{G}):=\{\hat{\chi}(\sigma),\ \sigma\in\mathrm{G}\}$. Then, the ring $R:=\mathbb{Z}_p[\mathrm{G}]_\chi$ is isomorphic to $\mathbb{Z}_p[\hat{\chi}(\mathrm{G})]$, which is the ring of integers of the unramified extension $\mathbb{Q}_p[\hat{\chi}(\mathrm{G})]$ of $\mathbb{Q}_p$. Thus, $R$ is a discret valuation ring. Moreover, the $R$-torsion of any $R$-module is equal to its $\mathbb{Z}_p$-torsion. Since $\chi\neq1$, Dirichlet unit theorem implies that the quotient $(\mathcal{O}_K^\times)_\chi/(\mu_K)_\chi$ is a free $R$-module of rank $1$. Let us define
\begin{equation*}
M:=p\#(\mathcal{O}_K^\times/\mathcal{E}_K)_\chi\#\mathrm{Cl}(\mathcal{O}_K)_\chi.
\end{equation*}
Let $\mu$, $U$ and $V$ be the images of $\mu_K$, $\mathcal{O}_K^\times$ and $\mathcal{E}_K$ in $K^\times/(K^\times)^M$ respectively. We deduce from above that, $U_\chi/\mu_\chi$ is a free $R/MR$-module of rank $1$. But since
\begin{equation}\label{tachfine}
U_\chi/V_\chi\simeq(\mathcal{O}_K^\times)_\chi/(\mathcal{E}_K)_\chi\simeq R/tR,
\end{equation}
for some divisor $t$ of $M$, there exists $\xi\in U_\chi$ giving an $R$-basis of $U_\chi/\mu_\chi$ and such that $\xi^t\in(\mathcal{E}_K)_\chi$. In particular $\xi$ has order $M$ in $K^\times/(K^\times)^M$. By Corollary \ref{sardine} there exists an ideal $\mathfrak{g}$ of $\mathcal{O}_k$ and an Euler system $\alpha:\mathcal{S}(\mathfrak{g})\longrightarrow k_\infty^\times$, such that the map $\kappa:=\kappa_\alpha$ defined by (\ref{kappa}) satisfies $\kappa(1)=\xi^t$. We define inductively classes $\mathfrak{c}_0,\ldots,\mathfrak{c}_i\in\mathrm{Cl}(\mathcal{O}_K)_\chi$, prime ideals $\lambda_1,\ldots,\lambda_i$ of $\mathcal{O}_K$, coprime with $\mathfrak{g}$, and ideals $\mathfrak{a}_0,\ldots,\mathfrak{a}_i$ of $\mathcal{O}_k$ such that $\mathfrak{c}_0=1$ and $\mathfrak{a}_0=1$. Let $i\geq0$, and suppose that $\mathfrak{c}_0,\ldots,\mathfrak{c}_i$ and $\lambda_1,\ldots,\lambda_i$ (if $i\geq1$) are defined. Then we set $\mathfrak{a}_i=\prod_{1\leq n\leq i}\ell_n$ (if $i\geq1$), where $\ell_n:=\lambda_n\cap\mathcal{O}_k$. Moreover,
\begin{itemize}
\item If $\mathrm{Cl}(\mathcal{O}_K)_\chi\neq<\mathfrak{c}_0,\ldots,\mathfrak{c}_i>_{\mathrm{G}}$, where $<\mathfrak{c}_0,\ldots,\mathfrak{c}_i>_{\mathrm{G}}$ is the G-module generated by $\mathfrak{c}_0,\ldots,\mathfrak{c}_i$, then we define $\mathfrak{c}_{i+1}$ to be any element of $\mathrm{Cl}(\mathcal{O}_K)_\chi$ whose image in $\mathrm{Cl}(\mathcal{O}_K)_\chi/<\mathfrak{c}_0,\ldots,\mathfrak{c}_i>_{\mathrm{G}}$ is nontrivial and is equal to a class $\mathfrak{c}$ which restricts to the generator $\mathfrak{c}'$ of $\mathrm{Gal}(L/K)$ in Theorem \ref{Chebotarev} applied to $\beta:=\kappa(\mathfrak{a}_i)_\chi$, the image of $\kappa(\mathfrak{a}_i)$ in $(K^\times/(K^\times)^M)_\chi$, and $A:=\mathrm{Cl}(\mathcal{O}_K)_\chi/<\mathfrak{c}_0,\ldots,\mathfrak{c}_i>_{\mathrm{G}}$. Also we let $\lambda_{i+1}$ be any prime ideal of $\mathcal{O}_K$ prime to $\mathfrak{g}$ and satisfying Theorem \ref{Chebotarev} with the same conditions.
\item If $\mathrm{Cl}(\mathcal{O}_K)_\chi=<\mathfrak{c}_0,\ldots,\mathfrak{c}_i>_{\mathrm{G}}$ then we stop.
\end{itemize}
This construction of our classes $\mathfrak{c}_j$ implies that the ideals $\ell_j:=\lambda_j\cap\mathcal{O}_k\in\mathcal{S}(\mathfrak{g})$. Let $m_i$ be the order of $\kappa(\mathfrak{a}_i)_\chi$ in $K^\times/(K^\times)^M$, and let $t_i:=M/m_i$. By the assertion $(iii)$ of Theorem \ref{Chebotarev} we have $\varphi_{\ell_{i+1}}(\kappa(\mathfrak{a}_i)_\chi)=ut_i\lambda_{i+1}$, for some $u\in\ \mathbb{Z}/M\mathbb{Z}[\mathrm{G}]_\chi^\times$. But $\mathfrak{a}_{i+1}=\mathfrak{a}_i\ell_{i+1}$. Thus 
\begin{equation}\label{thonrouge}
[\kappa(\mathfrak{a}_{i+1})_\chi]_{\ell_{i+1}}=ut_i\lambda_{i+1},
\end{equation}
thanks to (\ref{marrakech}). Now, by the definition of $t_{i+1}$, the fractional ideal of $\mathcal{O}_K$ generated by $\kappa(\mathfrak{a}_{i+1})_\chi$ is a $t_{i+1}$-th power. Thus, we must have $t_{i+1}\vert t_i$. Actually, we can say more. Indeed, there exist $\zeta\in\mu_K$ and $z\in K^\times$ such that $\kappa(\mathfrak{a}_{i+1})_\chi=\zeta z^{t_{i+1}}$. Therefore, (\ref{marrakech}) and (\ref{thonrouge}) imply
\begin{equation}\label{toumarte}
z\mathcal{O}_K=(\lambda_{i+1})^{ut_i/t_{i+1}}(\prod_{j=1}^i\lambda_j^{u_j})\mathfrak{b}^{M/t_{i+1}},
\end{equation}
where $u_j\in\mathbb{Z}_p[\mathrm{G}]$ for all $j\in\{1,\ldots,i\}$ and $\mathfrak{b}$ is a fractional ideal of $\mathcal{O}_K$. But we see from (\ref{tachfine}) that $t_0\vert\#(\mathcal{O}_K^\times/\mathcal{E}_K)_\chi$, and since $t_{i+1}\vert t_0$ the integer $M/t_{i+1}$ annihilates $\mathrm{Cl}(\mathcal{O}_K)_\chi$. The identity (\ref{toumarte}) then implies
\begin{equation}\label{idriss}
(t_i/t_{i+1})\mathfrak{c}_{i+1}\in<\mathfrak{c}_0,\ldots,\mathfrak{c}_i>_{\mathrm{G}}.
\end{equation}
Let $\dim(\chi):=[\mathbb{Q}_p[\hat{\chi}(\mathrm{G})]:\mathbb{Q}_p]$, then (\ref{idriss}) implies
\begin{equation*}
\#\mathrm{Cl}(\mathcal{O}_K)_\chi\,\vert\,\prod_{j=1}^i(t_{j-1}/t_j)^{\dim(\chi)}\,\vert\, t_0^{\dim(\chi)}=\#(\mathcal{O}_K^\times/\mathcal{E}_K)_\chi.
\end{equation*} 

\end{proof}
\noindent{\bf Proof of Theorem \ref{tresgras}}. Let the hypotheses and notation be as in Theorem \ref{tresgras}. Let $\Psi$ be the irreducible rational character of $\mathrm{G}$ such that $\chi\vert\Psi$. The formula (\ref{indice2}) may be written as follows
\begin{equation*}
\prod_{\chi'\vert\Psi}\#\mathrm{Cl}(K)_{\chi'}=\prod_{\chi'\vert\Psi}\#(\mathcal{O}_K^\times/\mathcal{E}_K)_{\chi'},
\end{equation*}
where $\chi'$ runs over the irreducible $\mathbb{Z}_p$-characters of $\mathrm{G}$ such that $\chi'\vert\Psi$. Moreover, the formula (\ref{division}) is satisfied for such characters $\chi'$ since $\chi\neq1$. This implies (\ref{gras}). 
\def\cprime{$'$} \def\cprime{$'$} \def\cprime{$'$}

\address{Laboratoire de math\'ematique\\ 16
Route de Gray\\ 25030 Besan\c con cedex\\ France}
{houkhaba@univ-fcomte.fr\\ sviguie@univ-fcomte.fr}

\end{document}